\documentclass[12pt]{amsart}

\usepackage{amsmath}
\usepackage{amssymb}

\newtheorem{theorem}{Theorem}[section]

\newtheorem{proposition}[theorem]{Proposition}

\newtheorem{remark}{Remark}

\def\Z{\mathbb Z}

\title[Evaluation of Polynomials over Finite Rings]{Evaluation 
of Polynomials over Finite Rings via Additive Combinatorics}
\author{Gyula K\'arolyi \and Csaba Szab\'{o}}

\begin{document}
\maketitle

\begin{abstract}
We give an improved polynomial bound on the complexity of 
the equation solvability problem, or more generally, of finding the value
sets of polynomials over finite nilpotent rings. Our proof depends
on a result in additive combinatorics, which may be of independent 
interest.
\end{abstract}

\section{Introduction}

\noindent
Several `classical' algebraic  problems are investigated from the 
computational perspective.
The equivalence problem and the equation solvability problem for
various algebraic structures has received increasing attention recently. Both
problems originate from the theory of rings and fields.  
The equivalence problem over a finite ring asks whether
or not two polynomials define the same function over the given ring.  
The equation solvability problem is  basically looking for 
the existence of a root or {for} the solution of an equation.  
It asks whether or not two polynomials over a ring have at least 
one substitution, where they attain the same value.
Both of these problems are decidable for finite rings: 
it is enough to substitute every element of the ring into the variables.

Early investigations into the computational complexity of the
equivalence problem were carried out by computer 
scientists at Syracuse University. In the early 1990s it was shown by
Hunt and Stearns \cite{HS} that the equivalence problem of a finite commutative
ring either can be decided in polynomial time or has
co-NP-complete complexity. Later Burris and Lawrence proved in \cite{bl} that
the same holds for rings in general: the problem can be decided in
polynomial time if the ring is nilpotent (this part is already found in 
\cite{HS}), and it is co-NP-complete otherwise.

The solvability problem has a somewhat shorter history.  Although the
borderline is again nilpotency, the proofs are more
complicated. Following the idea of Burris and Lawrence it is not hard
to verify that the solvability problem for non-nilpotent rings is NP-complete.  
Horv\'ath \cite{nilpHorvath} has proved that for nilpotent rings the
problem can be decided in polynomial time. 
More precisely, let $R$ be a nilpotent ring
of size $m$ and nilpotency class $\ell$, and let 
$f$ be a polynomial over $R$ in the (non-commuting) variables
$x_1,\ldots, x_n$. In \cite{nilpHorvath} it is shown that 
$O(||f||^t)$ substitutions suffice to decide, whether $f$ has a root or not.
Here $||f||$ represent the number of operations used to present $f$. 
The exponent $t=t(m,\ell)$, obtained in terms of the Ramsey-number for a
particular coloured hypergraph, is rather enormous, though. 

In the present paper, applying a recent result in additive combinatorics  
the exponent $t$ is reduced to $m\log m$; see Theorem \ref{main} for a  
more precise formulation. In particular, it is proved that the range of a
polynomial can be found using $O\left(n^{m(\ell-1)}\right)$ many 
substitutions.

Note that the equivalence and solvability  problems for finite groups 
has proved to be far more challenging. 
For results and detailed references see for example 
\cite{ExtGroup,a4hosza} and the most recent paper 
of Horv\'ath \cite{metaAbel}, 
where most of the existing results for equivalence and equation 
solvability are brought under a unified theory. 
For results about these problems on finite monoids and semigroups we refer to
\cite{AlmeidaVolkovGoldbergENG,KlimacoNPidcheck,seifszabo,tes};
concerning general algebras see e.g. \cite{TermeqSat}.

\section{Terminology}\label{sec:term}

\noindent
{\sl Polynomials.} \
Having a diverse readership in mind we introduce our notions in a somewhat
informal, albeit precise manner. Let $R$ be an arbitrary ring,
and $x_1,\dots,x_n$ symbols we think of as non-commuting variables. By a 
polynomial in these variables over $R$ we mean an expression that can be
constructed in a finite number of steps according to the following rules:
(i) every element of $\{x_1,\dots,x_n\}\cup R$ is a polynomial; (ii) if $f$
is a polynomial, then $-f$ is also a polynomial; (iii) if $f,g$ are
polynomials, then $f+g$ is also a polynomial; and (iv) if $f,g$ are
polynomials, then $fg$ is also a polynomial. The set of all such polynomials
we denote by $R[x_1,\dots,x_n]$.

Evaluating a polynomial $f\in R[x_1,\dots,x_n]$ at the point ${\bf c}=
(c_1,\dots,c_n)\in R^n$, or substituting ${\bf c}$ in $f$, means replacing 
each occurance of the variable $x_i$ in $f$ by the corresponding ring element
$c_i$, for every $1\le i\le n$, the result is an element of $R$ denoted by
$f({\bf c})$. This way $f$ defines a function from $R^n$ to $R$; two
polynomials are equivalent if they define the same function. For example,
for $a,b\in R$ the polynomials $(ax)b, a(xb), -(ax)(-b)\in R[x]$ are
equivalent, but may not be equivalent to $(ab)x$. One may even identify the
first three, but this is irrelevant to our purposes. The polynomials
$x_1(x_2x_1)$ and $x_1x_2$ are also equivalent over $\Z/2\Z$.

Each polynomial is equivalent to a `standard' polynomial, that is, to one
that can be written as a sum of monomials. More precisely, a monomial
is a polynomial that can be constructed using only the rules (i), (ii) 
and (iv), where (ii) is only applied in the last step, if at all. 
Suppressing the brackets each monomial can be represented in a canonical form
as $z_1\cdots z_t$ or $-z_1\cdots z_t$, where each 
$z_i\in \{x_1,\dots,x_n\}\cup R$. If two monomials have the same canonical
form, then they are clearly equivalent. A standard polynomial is one obtained
from monomials using merely the rule (iii). Once again, up to equivalence it
does not matter, in which order the additions are executed. 
\medskip

\noindent
{\sl The computational model.} \ 
In our model $R$ denotes a fixed finite ring, presented by its addition
and multiplication tables. Thus, its cardinality is $|R|=O(1)$. We may assume
that the table of additive inverses and the zero element are also given,
for they can be computed from the addition table in $O(1)$ time. 
In accordance with rules (i)--(iv),
the length $||f||$ of a polynomial $f\in R[x_1,\dots,x_n]$ is defined as
follows: $||f||=0$ if $f\in \{x_1,\dots,x_n\}\cup R$, and if 
$||f||,||g||$ are defined, then set $||-f||=||f||+1$ and $||f+g||=||fg||=
||f||+||g||+1$. Thus, $||f||$ is the number of operations used to define $f$,
that may be regarded as the complexity of evaluating $f$ at any single point of
$R^n$. Our goal is to measure the complexity of the problems mentioned in the
introduction in terms of $||f||$. 
\medskip

\noindent
{\sl Nilpotency.} \
In what follows, $R^{(i)}$ will denote the ideal
consisting of all finite sums of terms of the form $c_1\cdots c_i$,
where $c_1,\dots,c_i\in R$. The ring $R$ is nilpotent if there is a positive 
integer $i$ such that $R^{(i)}=0$; the smallest such $i$ is called the
nilpotency class of $R$. The standard example for a ring of nilpotency 
class $\ell$ is the ring of strictly upper triangular $\ell\times \ell$
matrices over an arbitrary field.

\section{A result in additive combinatorics}\label{sec:prel}

\noindent
A standard application of the pigeonhole principle gives that for any sequence
$g_1,g_2,\ldots,g_n$ of a finite group $G$ there is a subsequence
$g_{i_1},g_{i_2},\ldots,g_{i_t}$ with $t< |G|$ such that 
$g_{i_1}g_{i_2}\ldots g_{i_t}=g_1g_2\ldots g_n$. Here we need a more general
result that we can establish for finite abelian $p$-groups.

Let $p$ denote a prime and let $G$
be an arbitrary finite abelian $p$-group, written additively.
Then there exist unique positive integers $r$ and $\alpha_1\le 
\ldots\le \alpha_r$ such that 
$G\cong \Z/p^{\alpha_1}\Z\oplus\ldots\oplus\Z/p^{\alpha_r}\Z$.
For a finite set $H$ and a nonnegative integer $k$ we denote by $P_k(H)$
the set of all subsets $X$ of $H$ with $|X|\le k$. Given a function 
$\varphi\colon P_k(H)\to G$, associate with each subset $U$ of $H$ the value
$$\overline{\varphi}(U)=\sum_{X\subseteq U, |X|\le k}\varphi(X).$$
Thus, $\overline{\varphi}(H)$ is simply the total sum of the values of
$\varphi$ over $P_k(H)$.
\bigskip

\begin{theorem}
Let $H$ be a finite set, $k$ a nonnegative integer.
For any function 
$\varphi\colon P_k(H)\to \Z/p^{\alpha_1}\Z\oplus\ldots\oplus\Z/p^{\alpha_r}\Z$, 
there is a set $U\subseteq H$ such that 
$$|U|\le k\sum_{j=1}^r(p^{\alpha_j}-1)$$
and $\overline{\varphi}(U)=\overline{\varphi}(H)$.
\label{additive}
\end{theorem}

\begin{remark}
It is easily seen that the case $k=1$ of the above theorem 
is equivalent to Olson's classical result \cite{Olson} on the
Davenport constant of finite abelian $p$-groups.
\end{remark}

\begin{remark}
The bound on the cardinality of $U$ cannot be improved upon for any $k$. 
Assume that $|H|\ge k\sum_{j=1}^r(p^{\alpha_j}-1)$. Select pairwise disjoint
$k$-element subsets $H_{j,l}$ of $H$ for $1\le j\le r$, 
$1\le l< p^{\alpha_j}$. Denote by $e_j$ any generating element of the $j$th
direct summand $\Z/p^{\alpha_j}\Z$, and consider the function $\varphi$
that assigns $e_j$ to each set $H_{j,l}$ and 0 to any other element of 
$P_k(H)$. Then $\overline{\varphi}(H)=\sum_{j=1}^r(p^{\alpha_j}-1)e_j$, and
it is obvious that $\overline{\varphi}(U)=\overline{\varphi}(H)$ implies
$U\supseteq \cup H_{j,l}$.  
\end{remark}

\noindent
Our proof depends on the following result of Brink \cite{Brink},
which can be viewed as a generalization of Chevalley's well-known
theorem \cite{Chev} and its somewhat forgotten extension 
by Schanuel \cite{Scha}.
\bigskip

\begin{theorem}
Let $p$ denote a prime and let
$A_1,\ldots,A_n$ be nonempty subsets of $\Z$ such that
the natural homomorphism from $\Z$ to $\Z/p\Z$ restricted to $A_i$ is injective
for each $1\le i\le n$. Assume that the polynomials
$f_1,\ldots,f_r\in \Z[x_1,\ldots,x_n]$ satisfy
$$\sum_{i=1}^n(|A_i|-1)> \sum_{j=1}^r(p^{\alpha_j}-1)\deg f_j.$$
If the set
$\{{\bf a}\in A_1\times\ldots\times A_n\mid 
f_j({\bf a})\equiv 0 \pmod{p^{\alpha_j}},\ 1\le j\le r  \}$ is not empty,
then it has at least two different elements.
\label{brink}
\end{theorem}

\noindent{\it Proof of Theorem \ref{additive}.} 
Putting $|H|=n$, enumerate the elements
of $H$ as $h_1,\ldots,h_n$. 
We may assume that $n\ge k\sum_{j=1}^r(p^{\alpha_j}-1)+1$,
for the statement is valid with $U=H$ otherwise. 
For $1\le j\le r$ denote by $\varphi_j\colon P_k(H)\to \Z/p^{\alpha_j}\Z$ 
the $j$th coordinate function of $\varphi$, and
consider the polynomial $f_j\in (\Z/p^{\alpha_j}\Z)[x_1,\ldots,x_n]$
defined by
$$f_j({\bf x})=\sum_{\nu=0}^k
\left(\sum_{1\le i_1<\ldots<i_\nu\le n}\varphi_j(\{h_{i_1},\ldots,h_{i_\nu}\})
\prod_{\alpha=1}^\nu x_{i_\alpha}\right).$$
Thus, letting $f=(f_1,\ldots,f_r)$ we have 
$f({\bf 0})=\varphi(\emptyset)=\overline{\varphi}(\emptyset)$, 
$f({\bf 1})=\overline{\varphi}(H)$ and 
$$f({\bf x})=\overline{\varphi}(U)
\quad \hbox{\rm with}\quad  U=U({\bf x})=\{h_i\mid x_i=1\}$$
for ${\bf x}\in \{0,1\}^n$ in general.

Define sequences $I_s\subseteq \{1,\ldots, n\}$ and
${\bf a}{(s)}\in \{0,1\}^n$
recursively as follows. 
Put $I_0=\{1,\ldots, n\}$ and ${\bf a}{(0)}={\bf 1}$, then 
$$I_0=\{i\mid {\bf a}{(0)}_i=1\}\quad \hbox{\rm and} \quad
f({\bf a}(0))=\overline{\varphi}(H).$$
Assume that $I_s$ and ${\bf a}{(s)}$ are already defined such that
$$I_s=\{i\mid {\bf a}{(s)}_i=1\}\quad \hbox{\rm and} \quad
f({\bf a}(s))=\overline{\varphi}(H).$$
If $|I_s|> k\sum_{j=1}^r(p^{\alpha_j}-1)$, then
put $A_i=\{0,1\}$ for $i\in I_s$ and $A_i=\{0\}$ for $i\in
I_0\setminus I_s$. Note that ${\bf a}(s)\in A_1\times\ldots\times A_n$. 
For each $1\le j\le r$ the polynomial 
$g_j=f_j-\overline{\varphi}_j(H)\in (\Z/p^{\alpha_j}\Z)[x_1,\ldots,x_n]$ 
vanishes at ${\bf a}(s)$ and
$$\sum_{i=1}^n(|A_i|-1)=|I_s|>
k\sum_{j=1}^r(p^{\alpha_j}-1) \ge \sum_{j=1}^r(p^{\alpha_j}-1)\deg g_j.$$
It follows from Theorem \ref{brink} that there exists an element
$${\bf a}\in (A_1\times\ldots \times A_n)\setminus\{{\bf a}(s)\}$$ 
such that 
$g_i({\bf a})=0$ in $\Z/p^{\alpha_i}\Z$ for every $i=1,2,\dots,r$,
and accordingly $f({\bf a})=\overline{\varphi}(H)$.
Let 
$${\bf a}(s+1)={\bf a},\quad I_{s+1}=\{i\mid {\bf a}{(s+1)}_i=1\},$$
then $I_{s+1}$ is a proper subset of $I_s$.

This process must terminate at some $t\ge 1$ with 
$|I_t|\le k\sum_{j=1}^r(p^{\alpha_j}-1)$
meaning that the statement of the theorem is valid with
$$U=\{h_i\mid i\in I_t\}=U({\bf a}(t)).$$

\section{Nilpotent rings} 

\noindent
All the results claimed in Section 1 can be easily reduced to the 
following general statement.

\begin{proposition}
Let $R$ be an arbitrary finite ring, $g\in R[x_1,\ldots,x_n]$
a standard polynomial, written as a sum of monomials. If each 
such monomial contains at most $k$ different variables, 
then $O\left(n^{k|R|}\right)$ evaluations 
suffice to determine the range of $g$.
\label{folemma}
\end{proposition}

\begin{proof}
Assume first that the cardinality of $R$ is a power of a prime $p$.
That is, the additive group of $R$ is a finite abelian $p$-group:
$$R^+\cong \Z/p^{\alpha_1}\Z\oplus\ldots\oplus\Z/p^{\alpha_r}\Z.$$
Suppose that $g$ admits the value $v$: there are elements
$c_1,\ldots,c_n\in R$ such that $g(c_1,\ldots,c_n)=v$.
For each subset $X$ of $H=\{1,2,\ldots,n\}$ with $|X|\le k$, denote by
$g_X$ the sum of those monomial terms of $g$, which contain exactly the
variables $x_i$ with $i\in X$. Writing $\varphi(X)=g_X(c_1,\ldots,c_n)$ we have
$v=\sum_{X\in P_k(H)}\varphi(X)=\overline{\varphi}(H)$. 
Theorem \ref{additive} guarantees the existence of a subset $U$ of
$\{1,2,\ldots,n\}$ with
$$|U|\le k\sum_{j=1}^r(p^{\alpha_j}-1)\le k|R|$$
and $\overline{\varphi}(U)=\overline{\varphi}(H)=v$.
Putting $c'_i=c_i$ for $i\in U$ and $c'_i=0$ for $i\not\in U$ we have
$$g_X(c'_1,\ldots,c'_n)=
\left\{ \begin{array}{ll} g_X(c_1,\ldots,c_n)
& \mbox{if $X\subseteq U$,}\\
\hfill 0 \hfill & \mbox{otherwise.}\end{array}\right.$$
It follows that
\begin{eqnarray*}
g(c'_1,\ldots,c'_n)&=&\sum_{X\in P_k(H)}g_X(c'_1,\ldots,c'_n)=
\sum_{X\subseteq U}g_X(c_1,\ldots,c_n)\\
&=&\sum_{X\subseteq U}\varphi(X)=\overline{\varphi}(U)=\overline{\varphi}(H)=v.
\end{eqnarray*}
Consequently, to determine the range of $g$ it suffices to substitute
into $g$ only those $n$-tuples $(c_1,\ldots,c_n)$ in which all but at most
$k|R|$ elements are equal to 0. The number of such $n$-tuples is
$$\sum_{i=0}^{k|R|}\binom{n}{i}|R|^i\le
(k|R|+1)|R|^{k|R|}n^{k|R|}=O\left(n^{k|R|}\right),$$
where the implied constant only depends on $k$ and $R$.

Turning to the general case, assume that $R^+=G_1\oplus\ldots \oplus G_s$,
where $G_i$ is an abelian group of order $p_i^{\beta_i}$, the primes
$p_1,\ldots,p_s$ being pairwise different. Each set defined by
$$R_i=\{r\in R\mid p_i^{\beta_i}r=0\}$$
is an ideal in $R$ whose additive group must be identical to $G_i$. It
follows that $R=R_1\oplus\ldots \oplus R_s$, and there are natural
homomorphisms $\pi_i\colon R\to R_i$ such that $\pi_1+\ldots+\pi_s$ is
the identical map form $R$ to $R$.  
Thus, if the $n$-tuple
${\bf c}\in R^n$ is written as ${\bf c}={\bf c}_1+\ldots+{\bf c}_s$
with ${\bf c}_i\in (R_i)^n$, then $g({\bf c})=
g_1({\bf c}_1)+\ldots+g_s({\bf c}_s)$,
where $g_i$ is the polynomial over
$R_i$ obtained from $g$ by replacing each coefficient in $g$ by its image
under $\pi_i$. Accordingly, the value set of $g$ is the 
Minkowski-sum of the value
sets of the polynomials $g_i$, and in view of the first part of the proof
it can be found using
$$O\left(\prod_{i=1}^sn^{k|R_i|}\right)=
O\left(n^{k\sum_{i=1}^s|R_i|}\right)=O\left(n^{k|R|}\right)$$
substitutions.
\end{proof}

\begin{remark}\label{explicit}
As part of the preprocessing, one can compute in $O(1)$ time the factorization
$|R|=p_1^{\beta_1}\cdots p_s^{\beta_s}$, and then the ideals $R_i$ in the direct
decomposition of $R$. To compute the range of $g$ then it is enough to
evaluate $g$ on the elements of the small set
\[
S=\bigl\{ {\bf c}\in R^n : \bigl|\{ j:({\bf c}_i)_j\ne 0\}\bigr|\le 
kp_i^{\beta_i}\ \text{for every}\  1\le i\le s\bigr\}.
\]
According to the above proof, $|S|=O\left(n^{k|R|}\right)$.
\end{remark}

\begin{theorem}\label{main}
Let $R$ be a nilpotent ring of size $m$. Let $f\in R[x_1,x_2,\dots, x_n]$ 
be an $n$-variable polynomial. Then, it can be decided in
$O\left(||f||n^{m\log m}\right)$ time whether or not $f$ has a root in $R$. 
\end{theorem}

\begin{proof} 
Denote by $\ell$ the nilpotency class of $R$. Thus, $R^{(\ell)}=0$. Therefore  
$f$ is equivalent to a 
standard polynomial $\overline{f}\in R[x_1,x_2,\dots, x_n]$
in which every monomial term contains at most $\ell-1$ different variables.
By Proposition~\ref{folemma}, the range of $f$, which is the same as the range
of $\overline{f}$, can be found by evaluating
$f$ on a subset $S$ of $R^n$, of size $O\left(n^{(\ell-1)m}\right)$;
see Remark \ref{explicit} for an explicit description.
In particular, one can check whether $f$ admits the value 0, 
and it can be also decided if $f$ is identically zero, or not.
Note that the chain of ideals 
$R\triangleright R^{(2)}\triangleright \dots \triangleright R^{(\ell)}=0$
is strictly decreasing. 
As $|R^{(i)}/R^{(i+1)}|\geq 2$, it follows that $\ell-1\leq \log_2 m$.
\end{proof}

\begin{remark}
During the preparation of this article F\"oldv\'ari \cite{Foldvari} 
proved a similar result with an entirely different method.
His bound is $O\left(||f||^{m^{2\log m}\log^5m}\right)$. 
Note that one may assume that $||f||\ge n-1$, 
otherwise the set of variables that occur in $f$ can be restricted to
a proper subset of $\{x_1,\dots,x_n\}$.
If this is the case, then Theorem \ref{main} implies 
the bound $O\left(||f||^{m\log m}\right)$.
\end{remark}

The first author was supported by Bolyai Research Fellowship
and by National Research, Development and Innovation Office NKFIH Grant
K 120154.  
Research of the second author was supported by NKFIH Grant K 109185.

\end{document}